\documentclass{amsart}

\usepackage[utf8]{inputenc}
\usepackage{amssymb}
\usepackage{amsthm}
\usepackage {amsmath}
\usepackage{mathrsfs}
\usepackage{xcolor}

\newcommand {\N}{\mathbb{N}}
\newcommand {\R}{\mathbb{R}}
\newcommand {\Z}{\mathbb{Z}}
\newcommand{\Schw}{{\mathscr S}}
\newcommand{\wh}{\widehat}
\newcommand{\TD}{{\mathscr S'}}
\newcommand{\F}{{\mathscr F}}
\newcommand{\A}{{\mathcal A}}
\newcommand{\Rn}{\mathbb{R}^n}
\newcommand{\be}{\begin{equation}}
\newcommand{\ee}{\end{equation}}

\newtheorem{definition}{Definition}[section]
\newtheorem{theorem}{Theorem}[section]
\newtheorem{proposition}{Proposition}[section]

\newtheorem{ex}{Example}[section]
\newtheorem{lemma}{Lemma}[section]
\newtheorem{remark}{Remark}[section]
\numberwithin{equation}{section}
\newtheorem*{theorem*}{Theorem}

\begin{document}


\title[Weighted embedding theorems for radial spaces]{Weighted embedding theorems for radial Besov and Triebel-Lizorkin spaces}

\author{Pablo L.  De N\'apoli}
\address{IMAS (UBA-CONICET) and Departamento de Matem\'atica, Facultad de Ciencias Exactas y Naturales, Universidad de Buenos Aires, Ciudad Universitaria, 1428 Buenos Aires, Argentina}
\email{pdenapo@dm.uba.ar}

\author{Irene Drelichman}
\address{IMAS (UBA-CONICET) and Departamento de Matem\'atica, Facultad de Ciencias Exactas y Naturales, Universidad de Buenos Aires, Ciudad Universitaria, 1428 Buenos Aires, Argentina}
\email{irene@drelichman.com}

\author{Nicolas Saintier}
\address{ Instituto de Ciencias,  Universidad  Nacional de  Gral. Sarmiento, J. M. Gutierrez 1150, 
Los Polvorines, 1613 Provincia de Buenos Aires, Argentina and Departamento de Matem\'atica, Facultad de Ciencias Exactas y Naturales, Universidad de Buenos Aires, Ciudad Universitaria, 1428 Buenos Aires, Argentina}
\email{nsaintie@dm.uba.ar}

\thanks{Supported by ANPCyT under grant PICT 1675/2010, by CONICET under  grant PIP 1420090100230 and by Universidad de
Buenos Aires under grant 20020090100067. The authors are members of
CONICET, Argentina.}

\begin{abstract}
We study the continuity and compactness of embeddings for radial Besov and Triebel-Lizorkin spaces with weights in the Muckenhoupt class $A_\infty$. The main tool is a discretization in terms of an almost orthogonal wavelet expansion adapted to the radial situation.  
\end{abstract}


\maketitle

\section{Introduction}

Weighted embedding theorems for smooth function spaces have beeen studied by many authors, mainly because they are a fundamental tool in the variational analysis of some nonlinear partial differential equations, for instance of degenerate or singular elliptic equations.  It is therefore natural to study  embedding results in the framework of Triebel-Lizorkin and Besov spaces, since these include many of the classical functional spaces. In the unweighted case, a fundamental result in this context is the embedding theorem of Jawerth \cite{J} and Franke \cite{F}, which generalizes the classical Sobolev embedding theorem. 

Weighted Besov and Triebel-Lizorkin spaces have also been  studied by many authors under different assumptions on the weights (see e.g. \cite{B, BPT, VR}). Embeddings of Besov and Triebel-Lizorkin spaces with Muckenhoupt's $\mathcal{A}_\infty$ weights were studied by  Haroske and Skrzypczak in  \cite{HS1, HS2} and Meyries and Veraar in \cite{MV2} (see also  \cite{MV} for an earlier work by the same authors in the case of power weights).

On the other hand, it is well known, since the pioneering works of Ni \cite{Ni} and Strauss \cite{St}, that many embedding results can be improved when one considers subspaces of radial functions. More precisely, by restricting ourselves to the subspace of radial functions, we can  recover, for instance, compactness properties of embeddings that are in general non-compact due to the action of some non-compact group of transformations such as the group of translations in $\R^n$ (see, e.g. \cite{Lions}). 
Notice that compact embeddings are a  fundamental  feature for the success of  variational methods in PDE. 
In the case of weighted embedding theorems one can also obtain a wider range of exponents for the admissible power weights in the radial situation (see e.g. \cite{DD6}).  

In the case of unweighted radial subspaces of Besov and Triebel-Lizorkin spaces,  Sickel and  Skrzypczak  \cite{SiSk, SiSk3} and Sickel,  Skrzypczak and  Vybiral \cite{SiSk2}  obtained compactness of the related embeddings and an extension of Strauss' radial lemma. Quantitative information in terms of entropy numbers for the  embeddings  was obtained by K\"uhn, Leopold, Sickel and Skrzypczak in \cite{KLSS2}. In these papers, the main tool is an atomic decomposition adapted to the  radial situation. 

Early results on embeddings for weighted radial Besov and Triebel-Lizorkin spaces can be found in Triebel's book \cite[Section 6.5.2]{T3}, where the weights considered are of the  special forms $w_\alpha(x)=(1+|x|^2)^{\alpha/2}$ with $\alpha \in \R$, and $w^\beta(x)=e^{|x|^\beta}$ with $|x|\ge 1$ and $0<\beta\le 1$ (see also references therein). 
However, to the authors knowledge, results on weighted radial Besov and Triebel-Lizorkin spaces for other important classes of weights, such as power weights or, more generally, Muckenhoupt weights,  were still missing in the literature. The first two authors recently showed in \cite{DD4} that the approach used by Meyries and Veraar  \cite{MV} to obtain embedding theorems with power weights can be improved to obtain a better range of admissible exponents in the radial case. In this work we consider embedding theorems for radial subspaces of  Besov and Triebel-Lizorkin spaces with general   $\mathcal{A}_\infty$ weights. It is important to stress that in the latter case the functions considered are radially symmetric, but  the weights can be arbitrary. In the Triebel-Lizorkin case, we follow an argument by Meyries and Veraar \cite{MV2} to derive the embeddings from the Besov case, but this time restrict ourselves not only to radially symmetric functions but also to radially symmetric $\mathcal{A}_\infty$ weights (see the discussion in Section \ref{section-TL}). In both cases we obtain sufficient conditions for the continuity and compactness of the embeddings that improve with respect to the non-radial case. 

For our proof, instead of using the atomic decomposition for radial subspaces of Sickel and  Skrzypczak   \cite{SiSk}, we shall follow closely the approach used by Haroske and  Skrzypczak \cite{HS1, HS2} in the non-radial case, which is based on a discretization in terms of wavelet bases. To this end, we need a wavelet decomposition adapted to the radial situation, which we obtain by adapting arguments used by Epperson and  Frazier \cite{EF} in the unweighted radial case. We remark that this is not a wavelet decomposition in the traditional way, since the wavelets are localized near certain annuli instead of cubes. Hence, they have the advantage of being better adapted to the radial situation but have no translation structure and, more importantly, since they are not actual bases but rather frames, they do not characterize the (weighted) Besov and Triebel-Lizorkin spaces. In other words, they are useful to obtain sufficient conditions for the continuity and compactness of the embeddings, but cannot be used to prove sharpness of the conditions obtained.  Unfortunately, as far we know, there are no known orthogonal wavelet decompositions for radial functions except in dimension three (see, e.g. \cite{RR, CP}).

The rest of the paper is as follows. In Section 2 we recall some definitions and known properties of Besov and Triebel-Lizorkin spaces. Section 3 in devoted to the construction of the wavelet bases and the representation of  the weighted radial Besov and Triebel-Lizorkin spaces in terms of sequence spaces (Theorems \ref{wavelet-TL} and \ref{wavelet-characterization-theorem}). In Section 4 we prove our main theorem (Theorem \ref{main-theorem}) on sufficient conditions for the continuity and compactness of the embeddings for weighted radial Besov spaces and use it to analyze some important special examples. Finally, in Section 5 we obtain sufficient conditions for the continuity and compactness of the embeddings for Triebel-Lizorkin spaces with radial $\mathcal{A}_\infty$ weights (Theorem \ref{main-theorem-TL}) and an example in this case. 

\section{Weighted Besov and Triebel-Lizorkin spaces\label{subsec:defBF}}

First we recall some necessary definitions.  For classical references on 
Besov and Triebel-Lizorkin spaces see \cite{P, Tri83}.  For weighted versions see \cite{BPT, VR}.

\medskip

\begin{definition}[Construction of the Littlewood-Paley partition]\label{def:Phi}
Let $\varphi \in \Schw(\Rn)$ be such that
\begin{equation}\label{eq:propvarphi}
0\leq \wh{\varphi}(\xi)\leq 1, \quad  \xi\in \Rn, \qquad  \wh{\varphi}(\xi) = 1 \ \text{ if } \ |\xi|\leq 1, \qquad  \wh{\varphi}(\xi)=0 \ \text{ if } \ |\xi|\geq \frac32.
\end{equation}
Let $\wh{\varphi}_0 = \wh{\varphi}$, $\wh{\varphi}_1(\xi) = \wh{\varphi}(\xi/2) - \wh{\varphi}(\xi)$, and
\[\wh{\varphi}_\mu(\xi) = \wh{\varphi}_1(2^{-\mu+1} \xi) = \wh{\varphi}(2^{-\mu}\xi) - \wh{\varphi}(2^{-\mu+1}\xi),  
 \qquad \xi\in \Rn, \qquad  \mu\geq 1.\]
Then 
$$0\le \wh{\varphi}_\mu(\xi)\le 1, \quad  \wh{\varphi}_\mu(\xi) = 1 \ \text{ if  }\ \frac{3}{2}2^{\mu-1}\le|\xi|\le 2^\mu,
   \quad \text{supp}\,\wh{\varphi}_\mu\subset\Big\{2^{\mu-1}\le|\xi|\le \frac{3}{2}2^\mu\Big\}.$$ 
Let $\Phi$ be the set of all sequences $(\varphi_\mu)_{\mu\geq 0}$ constructed in the above way from a function $\varphi$ that satisfies \eqref{eq:propvarphi}.
\end{definition}

For  $\varphi$ as in the definition and $f\in \TD(\Rn)$ one sets
\[S_\mu f := \varphi_\mu * f = \F^{-1} [\wh{\varphi}_\mu \wh{f}],\]
which belongs to $C^\infty(\Rn)\cap \TD(\Rn)$. 
Since $\sum_{\mu\geq 0} \wh{\varphi}_\mu(\xi) = 1$ for all $\xi\in \Rn$, 
we have $\sum_{\mu\geq 0} S_\mu f =  f$ in the sense of distributions.

\medskip

Given a weight $w$, that is a non-negative locally integrable function on $\Rn$, and a real number $p\in [1,+\infty]$, we denote by $L^p(\R^n,w)$ the weighted Lebesgue space defined as the space of those measurable functions $f:\Rn\to \R$ such that 
$$ \|f\|^p_{L^p(\Rn,w)}:= \int_{\Rn} |f|^p\,w(x)\,dx <\infty $$ 
 if $1\le p<+\infty$, and $\|f\|_{L^\infty(\Rn, w)}=\|f\|_{L^\infty(\Rn)}$. 

Let us recall that, for $1<p< + \infty$, the Muckenhoupt class $\A_p$ is the class of weights  $w$ for which the maximal Hardy-Littlewood operator is bounded from $L^p(\Rn,w)$ to itself, and that it can be characterized by the condition
$$
\left( \frac{1}{|B|} \int_B w \right) \left( \frac{1}{|B|} \int_B w^{1-p'}\right)^{p-1} \le C
$$
for all balls $B \subseteq \Rn$, where the constant $C$ depends on $w$ but is independent of $B$. On the other hand, we say that $w\in \A_1$ if $Mw(x)\le C w(x)$ a.e., and we set $\A_\infty=\bigcup_{p\ge 1} \A_p$. We refer to \cite{CF} for a detailed account of these weights. 

\medskip

Given real numbers $p,q\in [1,\infty]$, $s\in\R$ and a weight $w\in \mathcal{A}_\infty$, we can  define following \cite{B} the weighted Besov and Triebel-Lizorkin spaces $B_{p,q}^s (\Rn,w)$ and $F_{p,q}^s (\Rn,w)$ in the following way

\begin{definition}\label{def:Besov}
The (inhomogeneous) Besov space $B_{p,q}^s(\Rn,w)$ is defined as the space of all  $f\in {\mathscr S}'(\Rn)$ for which 
\[ \|f\|_{B_{p,q}^s (\Rn,w)} := \Big( \sum_{\mu\ge 0} 2^{q\mu s} \|S_\mu f\|_{L^p(\Rn,w)}^q \Big)^{1/q}
                                < \infty.\]
with the usual modifications for $q = \infty$.
\end{definition}

\begin{definition}\label{def:Triebel-Lizorkin}
Assume that $p<\infty$. 
The (inhomogeneous)  Triebel-Lizorkin space $F_{p,q}^s(\Rn,w)$ is defined as the space of all 
$f\in {\mathscr S}'(\Rn)$ for which
\[ \|f\|_{F_{p,q}^s (\Rn,w)}    := \Big\|\Big(\sum_{\mu\geq 0} 2^{q\mu s}|S_\mu f|^q\Big)^{1/q}\Big\|_{L^p(\Rn,w)}<\infty.\]
with the usual modifications for $q = \infty$.
\end{definition}

\begin{remark}
\begin{enumerate}
\item It can be proved that these definitions do not depend on the choice of the particular $\varphi$ in (\ref{eq:propvarphi}), see e.g. \cite{BPT}. 
 \item  The corresponding homogeneous spaces denoted by $\dot B_{p,q}^s(\Rn,w)$ and $\dot F_{p,q}^s(\Rn,w)$ are defined in a similar way with the sum running over $\Z$ with appropriate modifications of the partition of unity. Observe that $\|f\|_{B_{p,q}^s (\Rn,w)}=0$ if and only if $\text{supp}\,\hat f=\{0\}$, i.e., $f$ is a polynomial. 
 For this reason it is usual to consider instead the quotient spaces $\dot B_{p,q}^s(\Rn,w)/\mathcal{P}$ and $\dot F_{p,q}^s(\Rn,w)/\mathcal{P}$ where $\mathcal{P}$ is the space of polynomials.
\item If $w\equiv 1$, we write $B_{p,q}^s(\Rn)$ instead of $B_{p,q}^s(\Rn,w)$ and $F_{p,q}^s(\Rn)$ instead of $F_{p,q}^s(\Rn,w)$. 
\end{enumerate} 
\end{remark} 

\medskip

The group $O(n)$ of $\Rn$ acts on $\Schw(\Rn)$ by $(\sigma,\phi)\in O(n)\times \Schw(\Rn)\to \sigma\phi\in \Schw(\Rn)$ with $\sigma\phi(x):=\phi(\sigma^{-1}x)$. 
Then, for any $f,\phi\in \Schw(\Rn)$ and $\sigma\in O(n)$ there holds that $(\sigma.f,\phi)_{L^2}=(f,\sigma^{-1}\phi)_{L^2}$. We thus define the action of $O(n)$ on $\TD(\Rn)$ by $(\sigma,f)\in O(n)\times \TD(\Rn)\to \sigma.f\in \TD(\Rn)$ with 
\begin{equation}\label{Action}
 (\sigma.f,\phi):= (f,\sigma^{-1}\phi) \qquad \text{ for any } \phi\in\Schw(\Rn). 
\end{equation}
This motivates our next definition:

\begin{definition}\label{def:RadialDistrib}
We say that a tempered distribution $f\in \TD(\Rn)$ is radial if $\sigma.f=f$ for any $\sigma\in O(n)$ where $\sigma.f$ is defined by \eqref{Action}. 
\end{definition} 

\noindent The Besov and Triebel-Lizorkin spaces of radial distributions will be denoted by $RB_{p,q}^s(\Rn,w)$ and $RF_{p,q}^s(\Rn,w)$, respectively. The following embeddings between these spaces are elementary and follow from the corresponding non-radial situation (see \cite[Theorem 2.6]{B}).
 
\begin{theorem}
\label{elementary-embeddings}
Let $w\in A_\infty$. Then
\begin{enumerate}
\item For all $1\le q_1 \le q_2 \le \infty$ and $s\in \R$ one has
$$
RB_{p,q_1}^s(\R^n, w) \hookrightarrow RB_{p,q_2}^s(\R^n, w), \quad p\in [1, \infty];
$$
$$
RF_{p,q_1}^s(\R^n, w) \hookrightarrow RF_{p,q_2}^s(\R^n, w), \quad p\in [1, \infty].
$$
\item For all $q_1, q_2 \in [1, \infty]$, $s\in \R$ and $\varepsilon >0$ one has
$$
RB_{p,q_1}^{s+\varepsilon}(\R^n, w) \hookrightarrow RB_{p,q_2}^s(\R^n, w), \quad p\in [1, \infty];
$$
$$
RF_{p,q_1}^{s+\varepsilon}(\R^n, w) \hookrightarrow RF_{p,q_2}^s(\R^n, w), \quad p\in [1, \infty].
$$
\item  For all $q \in [1, \infty]$, $s\in \R$ and $p\in [1, \infty)$ one has
$$
RB_{p,\min\{p,q\}}^{s}(\R^n, w) \hookrightarrow RF_{p,q}^s(\R^n, w)  \hookrightarrow RB_{p,\max\{p,q\}}^{s}(\R^n, w).
$$
\end{enumerate}
\end{theorem}

We now state a weighted version due to \cite{B} of the continuity of Peetre maximal function  originally defined in \cite{P2}. 

Let $a>0$ and $\{ \phi_\mu\}_{\mu\ge 0} $ be a sequence of functions in $\Schw(\Rn)$ such that 
$$  \hbox{supp}\, \widehat{\phi_\mu} \subset \{ 2^{\mu-a} \leq |\xi| \leq 2^{\mu+a}\}, $$
and
$$ |D^\alpha \widehat{\phi_\mu}(\xi)| \leq C_n 2^{-\mu|\alpha|}  \quad\text{for all } \mu\ge 0,\,\alpha\in\N^d,\, \xi\in\Rn. $$
This holds e.g. if $\widehat{\phi_\mu}(\xi)=\widehat\phi_1(2^{-\mu}\xi)$. 
For a given $\lambda>0$ the Peetre maximal functions of $f \in \Schw^\prime(\Rn)$ are 
\begin{equation}\label{DefPeetre}   
 \phi_{\mu,\lambda}^* f(x) = \phi_\mu^* f(x) = \sup_{y \in \R^n} \frac{|\phi_\mu * f(x-y)|}{(1+2^\mu|y|)^{\lambda}},
\quad   x\in\Rn,\,\mu\ge 0.
\end{equation}

\begin{theorem}\cite[Section 5]{B}
\label{peetre-maximal}
Let $r_0=\inf\{ r : w\in A_r \}$.  
\begin{enumerate}
\item[i)] If $\lambda>\max\{\frac{n r_0}{p},\frac{n}{q}\}$ then
\begin{equation}\label{ContMaxPeetre} 
 \left \| \left( \sum_{\mu\ge 0} [2^{\mu s}\phi_\mu^* f(x)]^q \right)^\frac{1}{q} \right \|_{L_p(\Rn,w)} \leq C \| f \|_{ F_{p,q}^{s}(\Rn, w)}
   \quad \text{for all } f \in \Schw^\prime(\Rn). 
\end{equation} 
\item[ii)] If $\lambda> \frac{nr_0}{p}$ then
$$ 
\left( \sum_{\mu\ge 0} [2^{\mu s} \| \phi_\mu^* f \|_{L_p(\Rn, w)}]^q \right)^\frac{1}{q} \leq C \| f \|_{  B_{p,q}^{s}(\Rn, w)} 
 \quad \text{for all } f \in \Schw^\prime(\Rn). 
 $$
\end{enumerate}
\end{theorem}

\section{Construction of radial wavelets for weighted Besov and Triebel-Lizorkin spaces}

In this section we develop a suitable wavelet decomposition adapted to the weighted radial situation. Our starting point is the construction of radial wavelets of Epperson and Frazier \cite{EF}.

Let $\Phi,\Psi,\varphi, \psi \in \Schw (\Rn)$ be radial functions such that 
$$ \text{supp}\,\hat\Phi,\,\text{supp}\,\hat\Psi \subset \{|\xi|\le 1\},\qquad |\hat\Phi(\xi)|,\, |\hat\Psi(\xi)|\ge c>0 \text{ if } |\xi|\le 5/6, $$ 
$$\hbox{supp} \; \widehat{\varphi},\widehat{\psi} \subset \{ 1/4  < |\xi| < 1\}, \qquad |\widehat{\varphi}|,|\widehat{\psi}|\geq c >0 
\text{ if } 3/10 \leq |\xi| \leq 5/6, $$
and 
$$ \overline{\hat\Phi}(\xi)\hat\Psi(\xi) 
 + \sum_{\mu\ge 1} \; \overline{\widehat{\varphi_\mu}}(\xi) \; {\widehat{\psi_\mu}}(\xi) = 1 
\quad  \hbox{for} \; \xi \neq 0. $$
where $\varphi_{\mu}(x)= 2^{\mu n} \varphi(2^\mu x)$ and $\psi_{\mu}(x)= 2^{\mu n} \psi(2^\mu x)$. 
We then define a family of functions $(\varphi_{\mu k})_{\mu\ge 0,k\ge 1}$  by 
$$ \varphi_{\mu k} = 
\begin{cases} 
\left( \frac{\displaystyle 2^{(\mu(n-2)+1)}}{\displaystyle  j^n_{\nu,k} J_{\nu+1}^2(j_{\nu,k})
\omega_{n-1}} \right)^{1/2} \varphi_\mu * d\sigma_{2^{-\mu} j_{\nu,k}} \quad \text{ for } \mu\ge 1, \\
\left( \frac{\displaystyle 2}{\displaystyle j^n_{\nu,k} J_{\nu+1}^2(j_{\nu,k})
\omega_{n-1}} \right)^{1/2} \Phi * d\sigma_{2^{-\mu} j_{\nu,k}} \quad \text{ for } \mu =0,
\end{cases} 
$$ 
where $d\sigma_t$ denotes the (unnormalized) surface Lebesgue measure on the sphere of radius $t$ in $\Rn$, $\omega_{n-1}$ the surface of the unit sphere, and 
$$ 0 < j_{\nu,1} < j_{\nu,2} < \ldots < j_{\nu,k} < \ldots  $$
denote the positive zeros of the Bessel function $J_\nu$ with $\nu=(n-2)/2$.  
We define in a similiar way the functions $(\varphi_{\mu k})_{\mu\ge 0,k\ge 1}$.  
Then the Epperson-Frazier  wavelet expansion for a radial  distribution $f \in \Schw^\prime(\R^n)$ is given by
$$ f = \sum_{\mu \ge 0} \sum_{k\ge 1} \; \langle f, \varphi_{\mu,k} \rangle \psi_{\mu,k}. $$

Epperson and Frazier were  able to characterize the membership of $f$ to  (unweighted) Besov or Triebel-Lizorkin spaces in terms of the wavelet coefficients $\langle f, \varphi_{\mu,k} \rangle$. Our purpose in this section is to show that analogous results hold for the weighted version of these spaces when the weight belongs to the $\mathcal{A}_\infty$ class. 

\medskip

We consider the annuli $A_{\mu,k}$, $\mu\ge 0$, $k\ge 1$, defined by 
$$ A_{\mu,k} = \{x\in\Rn,\,2^{-\mu}j_{\nu,k-1}\le |x|\le 2^{-\mu}j_{\nu,k}\} \quad\text{ with } j_{\nu,0}=0, $$ 
and denote by $\chi_{\mu,k}:=|A_{\mu,k}|^{-1/2}\chi_{A_{\mu,k}}$ its $L^2$-normalized characteristic function. 
Given real numbers $p,q\in [1,\infty]$, $s\in\R$ and a weight $w\in \mathcal{A}_\infty$ we let $b_{p,q}^{s}(w)$ and 
$ f_{p,q}^s(w)$ be the spaces of sequences of complex numbers $\lambda:=(\lambda_{\mu,k})_{\mu,k}$ such that 
$$  
\|\lambda\|_{ b_{p,q}^{s}(w)}:= 
\Big( \sum_{\mu\ge 0}\Big\| \sum_{k\ge 1} 2^{\mu s}|\lambda_{\mu,k}|\chi_{\mu,k}\Big\|_{L_p(\Rn,w)}^q \Big)^\frac{1}{q}<\infty,  
$$ 
and
$$  \|\lambda\|_{ f_{p,q}^s(w)}:= 
\Big\| \Big ( \sum_{\mu\ge 0}  \sum_{k\ge 1} [2^{\mu s}|\lambda_{\mu,k}|\chi_{\mu,k}]^q \Big)^\frac{1}{q} \Big\|_{L^p(\Rn,w)}
<\infty $$
respectively, with the usual modifications if $q=\infty$.

\medskip

Our first result is the following:

\begin{theorem} 
\label{wavelet-TL}
Let $p,q\in [1, \infty]$ and $w\in \mathcal{A}_\infty$. Then, the operators 
$$ S:f\in  RF_{p,q}^{s}(\Rn, w) \to (\langle f, \varphi_{\mu,k} \rangle)_{\mu,k}\in  f_{p,q}^{s}(w) $$ 
and 
$$ T:\lambda\in  f_{p,q}^{s}(w)  \to \sum_{\mu\ge 0}\sum_{k\ge 1} \lambda_{\mu,k}\psi_{\mu,k} \in  RF_{p,q}^{s}(\Rn, w) $$ 
are bounded, and the composition  $T \circ S$ is the identity on $RF_{p,q}^{s}(\Rn, w)$. In particular, $\|f\|_{RF_{p,q}^s(w)} \simeq \|S(f)\|_{f_{p,q}^s(w)}$. 
\end{theorem} 

\begin{remark}
The same type of result holds for homogeneous spaces with the usual modification, namely, by summing over $\mu\in\Z$ and suppressing $\Phi$ and $\Psi$. 
\end{remark}

\begin{proof} 
The case $w\equiv 1$ corresponds to  \cite[Theorem 2.1 and 2.2]{EF}. Since the proof in the general case is a  modification of those results, we  sketch it indicating where changes are needed. These mainly concern the continuity of the Peetre maximal function and of the Hardy-Littlewood maximal function for sequences of functions. 

Concerning the continuity of $S$ we have as in the proof of  \cite[Theorem 2.1]{EF} that for any $\mu\ge 0$  and $ \lambda>0$, 
$$ \sum_{k\ge 1} (2^{\mu s}| \langle f, \varphi_{\mu,k} \rangle| \chi_{\mu,k}(x))^q 
   \le C_\lambda 2^{\mu sq} |\varphi_\mu^*f(x)| ^q\, \text{ a.e. } $$ 
where $\varphi_\mu^*$ is the Peetre maximal function as defined in (\ref{DefPeetre}) for  $\lambda>0$.  
According to Theorem \ref{peetre-maximal} we  obtain, taking $\lambda $ big enough, that 
$$  \|S(f)\|_{ f_{p,q}^{s}(w) }
    \le C \Big\| \Big( \sum_{\mu \ge 0} 2^{\mu sq} |\varphi_\mu^*f(x)| ^q \Big)^\frac{1}{q}\Big\|_{L_p{(\Rn,w)}}  \\
    \le  C \| f \|_{ RF_{p,q}^{s}(w)}. 
$$

Concerning the continuity of $T$, fix $\lambda\in  f_{p,q}^{s}(w) $ and let 
$f=\sum_{\mu\ge 0}\sum_{k\ge 1} \lambda_{\mu,k}\psi_{\mu,k}$. 
Then for any $\eta\in (0,1]$ such that $p/\eta,q/\eta>1$ we have as in \cite{EF} that 
\begin{align*}
\|f\|_{ RF_{p,q}^{s}(w) } &  = \Big\|\Big ( \sum_{\mu\ge 0} (2^{\mu s} |\varphi_\mu * f|)^q \Big)^\frac{1}{q} \Big\|_{L^p(\Rn, w)}\\
 &\le C 
\Big \|  \Big ( \sum_{\mu\ge 0} \Big ( M \Big( \sum_{k\ge 1} (2^{\mu s} |\lambda_{\mu,k}| \chi_{\mu,k} )^\eta \Big)\Big)^\frac{q}{\eta} \Big)^\frac{\eta}{q} \Big\|_{L^\frac{p}{\eta}(\Rn,w)}^\frac{1}{\eta}, 
\end{align*}
where $M$ is the Hardy-Littlewood maximal function. According to \cite[Theorem 3.1]{AR} or \cite[Theorem 1]{Ko}, the vector-valued maximal function between weighted spaces 
$$ M:(f_\mu)_\mu\in L^{\alpha}(\ell_\beta,w) \to (Mf_\mu)_\mu\in  L^{\alpha}(\ell_\beta,w) $$ 
is continuous when the weight $w$ belongs to the $\A_\alpha$ class with $1<\alpha,\beta<\infty$. Here $L^\alpha(\ell^\beta)$ denotes the space of sequences of locally integrable functions $(f_\mu)_\mu$ such that
$$ \| (f_\mu)_\mu\|_{L^\alpha(\ell_\beta,w)}^\alpha:= 
   \int_{\R^n} \Big(\sum_\mu |f_\mu|^\beta\Big)^\frac{\alpha}{\beta}\,w\,dx <\infty. $$  
Since $w\in \mathcal{A}_p$, taking $\eta$ small enough to have $p/\eta>r_0:=\inf\,\{r : w\in\A_r\}$ we get that $w\in\A_\frac{p}{\eta}$. 
It follows that $M: L^{p/\eta}(\ell_{q/\eta},w) \to L^{p/\eta}(\ell_{q/\eta},w) $ is continuous. 
We thus obtain 
$$ \|f\|_{ RF_{p,q}^s(w) } \le 
C \Big\| \Big( \sum_{\mu\ge 0} \Big( \sum_{k\ge 1} (2^{\mu s}|\lambda_{\mu,k}|\chi_{\mu,k})^\eta\Big)^\frac{q}{\eta} \Big)^\frac{\eta}{q} \Big\|_{L_\frac{p}{\eta}(\Rn, w)}^\frac{1}{\eta}.  $$ 
Since for given $\mu$ the annuli $A_{\mu,k}$, $k\ge 1$, are essentially disjoint we obtain 
$$ 
\|f\|_{RF_{p,q}^{s}(w) } \le 
C \Big\| \Big( \sum_{\mu\ge 0}  \sum_{k\ge 1} (2^{\mu s}|\lambda_{\mu,k}|\chi_{\mu,k})^q  \Big)^\frac{1}{q} \Big\|_{L_p(\Rn,w)}
= C \|\lambda\|_{ f_{p,q}^{s}(w)}. 
$$
\end{proof} 

The analogous statement  for weighted Besov spaces reads as follows: 

\begin{theorem} 
Let $p,q\in [1, \infty]$ and $w\in \mathcal{A}_\infty$. Then, the operators 
$$ S:f\in  RB_{p,q}^{s}(\Rn, w) \to (\langle f, \varphi_{\mu,k} \rangle)_{\mu,k}\in  b_{p,q}^{s}(w) $$ 
and 
$$ T:\lambda\in  b_{p,q}^{s}(w)  \to \sum_{\mu\ge 0}\sum_{k\ge 1} \lambda_{\mu,k}\psi_{\mu,k} \in  RB_{p,q}^{s}(\Rn, w) $$ 
are bounded, and the composition  $T \circ S$ is the identity on $RB_{p,q}^{s}(\Rn, w)$. In particular, $\|f\|_{RB_{p,q}^s(w)} \simeq \|S(f)\|_{b_{p,q}^s(w)}$. The same result holds also for the homogeneous version of these spaces. 
\label{wavelet-characterization-theorem}
\end{theorem} 

\begin{proof}
The unweighted case $w=1$ corresponds to \cite[Theorems 5.1 and 5.2]{EF}.

For the continuity of $S$, as in the proof of the  previous theorem,  we obtain that
$$ \sum_{k\ge 1} 2^{\mu s}| (\langle f, \varphi_{\mu,k} \rangle)| \chi_{\mu,k}(x)
   \le C2^{\mu s} |\varphi_\mu^*f(x)| \, \text{ a.e. } $$ 
where $\varphi_\mu^*$ is the Peetre maximal function  for a given $\lambda>0$. Taking $\lambda$ big enough and using  Theorem \ref{peetre-maximal} we have
$$
\|S(f)\|_{b_{p,q}^s(w)} \le C \Big( \sum_{\mu \ge 0} 2^{\mu s q} \|\varphi_\mu^* f\|_{L^p(\Rn, w)}^q \Big)^\frac{1}{q} \le C \|f\|_{RB_{p,q}^s(w)}.
$$

For the continuity of $T$,  fix $\lambda\in  b_{p,q}^{s}(w) $ and let 
$f=\sum_{\mu\ge 0}\sum_{k\ge 1} \lambda_{\mu,k}\psi_{\mu,k}$. Then, arguing similarly as in the Triebel-Lizorkin case we see that for any $\mu\ge 0$, 
\begin{equation*}
\begin{split} 
\|\varphi_\mu * f\|_{L^p(\Rn, w)} 
& \le C \sum_{\nu= \mu-1}^{\mu +1} \Big\| \Big ( M \Big ( \sum_{k\ge 1} |\lambda_{\nu, k}|^\eta \chi_{\nu, k}^\eta \Big )  \Big )^\frac{1}{\eta}\Big \|_{L^p(\Rn, w)}  \\
& = C \sum_{\nu= \mu-1}^{\mu +1} \Big\| M \Big ( \sum_{k\ge 1} |\lambda_{\nu, k}|^\eta \chi_{\nu, k}^\eta \Big )  \Big \|_{L^{p/\eta}(\Rn, w)}^{1/\eta}.
\end{split}
\end{equation*} 
Since $w\in \mathcal{A}_\infty$, setting  as before $r_0:=\inf \{r: w\in \mathcal{A}_r\}$ and taking $\eta$ small enough to have $r_0<p/\eta$ we get that $w\in\A_\frac{p}{\eta}$ so that 
the  maximal operator $M:L^{p/\eta}(\Rn,w)\to L^{p/\eta}(\Rn,w)$ is continous. Then 
\begin{equation*}
\begin{split} 
  \Big\| M \Big ( \sum_{k\ge 1} |\lambda_{\nu, k}|^\eta \chi_{\nu, k}^\eta \Big )  \Big \|_{L^{p/\eta}(\Rn, w)}^{1/\eta} 
  &\le C \Big\|  \sum_{k\ge 1} |\lambda_{\nu, k}|^\eta \chi_{\nu, k}^\eta   \Big \|_{L^{p/\eta}(\Rn, w)}^{1/\eta}  \\
  & = C \Big\|  \sum_{k\ge 1} |\lambda_{\nu, k}| \chi_{\nu, k}   \Big \|_{L^p(\Rn, w)},
\end{split}
\end{equation*} 
where we have used the fact that for  given $\nu$, the annuli $A_{\nu,k}$ are essentially disjoint. 
We deduce that  
\begin{align*}
\|f\|^q_{ RB_{p,q}^{s}(w) } 
&  =   \sum_{\mu \ge 0} 2^{\mu s q} \|\varphi_\mu * f\|_{L^p(\Rn, w)}^q  \\
& \le C\sum_{\mu\ge 0} 2^{\mu sq} \sum_{\nu=\mu-1}^{\mu+1} 
     \Big\| \sum_{k\ge 1} |\lambda_{\nu,k}|\chi_{\nu,k}\Big\|^q_{L^p(\Rn,w)} \\
& \le C\sum_{\mu\ge 0} 2^{\mu sq}  \Big\| \sum_{k\ge 1} |\lambda_{\nu,k}|\chi_{\nu,k}\Big\|^q_{L^p(\Rn,w)} \\     
& = C\|\lambda\|_{b_{p,q}^s(w)}^q
\end{align*}
\end{proof}

\section{Continous and compact embeddings of weighted radial Besov spaces}

In this section we use Theorem \ref{wavelet-characterization-theorem}
 to obtain sufficient  conditions for the continuity and compactness of the embeddings of weighted radial Besov spaces, and apply these results to some relevant examples.

\begin{theorem}
\label{main-theorem} Let $p_1, p_2, q_1, q_2\in[1, \infty]$ and $w_1,w_2$ be $\mathcal{A}_\infty$-weights.   
There is a continuous embedding $RB_{p_1,q_1}^{s_1}(\Rn, w_1)\to RB_{p_2,q_2}^{s_2}(\Rn, w_2)$ provided that
\begin{equation}\label{CNS-cont-sequ-space}
\left\{2^{-\mu(s_1-s_2)} \left\|\left\{\frac{w^2_{\mu k}}{w^1_{\mu k}}\right\}_k\right\|_{\ell_{p*}}\right\}_\mu \in \ell_{q*}
\end{equation}

where 
$$ w^1_{\mu k} = \|\tilde\chi_{A_{\mu k}}\|_{L^{p_1}(\Rn, w_1)},\quad 
  w^2_{\mu k} = \|\tilde\chi_{A_{\mu k}}\|_{L^{p_2}(\Rn, w_2)}, $$ 
and 
$$ \frac{1}{p^*}:=\left(\frac{1}{p_2}-\frac{1}{p_1}\right)_+,\quad 
   \frac{1}{q^*}:=\left(\frac{1}{q_2}-\frac{1}{q_1}\right)_+. $$ 
The embedding is compact provided that \eqref{CNS-cont-sequ-space} holds and
moreover 
\begin{equation*}
\begin{split} 
 & \lim_{\mu\to+\infty} 2^{\mu(s_2-s_1)}  \left\|\left\{\frac{w^2_{\mu k}}{w^1_{\mu k}}\right\}_k\right\|_{\ell_{p*}}= 0
    \quad \text{ if } q^*=\infty \\
 & \lim_{|k|\to +\infty} \frac{w^1_{\mu k}}{w^2_{\mu k}} = \infty \quad \text{ for all $\mu\ge 0 $ if } p^*=\infty.  
\end{split} 
\end{equation*} 
\end{theorem}

\begin{proof} 
By Theorem \ref{wavelet-characterization-theorem} it suffices to study the embedding of the corresponding sequence spaces 
$$  b_{p_1,q_1}^{s_1}(w_1)\to b_{p_2,q_2}^{s_2}(w_2) $$
that is, using the notation of \cite[section 3]{KLSS}, 
$$ \ell_{q_1}(2^{\mu s_1}\ell_{p_1}(w_1)) \to \ell_{q_2}(2^{\mu s_2}\ell_{p_2}(w_2)). $$ 
Notice that the continuity of this  embedding is equivalent to the continuity of the embedding 
$$ \ell_{q_1}(2^{\mu(s_1-s_2)}\ell_{p_1}(\frac{w_1}{w_2})) \to \ell_{q_2}(\ell_{p_2}). $$ 
Indeed 
$$ \|\lambda\|_{\ell_{q_2}(2^{\mu s_2}\ell_{p_2}(w_2))}
 = \|\tilde \lambda\|_{\ell_{q_2}(\ell_{p_2})},\quad \text{ with } \quad 
  \tilde \lambda_{\mu k}= \lambda_{\mu k}w^2_{\mu k}2^{\mu s_2}. $$ 
We can rewrite this embedding using the notation of \cite{KLSS} as 
$$ \ell_{q_1}(\beta_\mu \ell_{p_1}(w)) \to \ell_{q_2}(\ell_{p_2})
\quad \text{ with }\quad  
\beta_\mu=2^{\mu(s_1-s_2)},\, w=(w_{\mu k})_{\mu k},\, w_{\mu k}=\frac{w^1_{\mu k}}{w^2_{\mu k}}.   $$ 
According to \cite[Theorem 3.1]{KLSS}, this embedding is continuous if and only if 
$$ (\beta_\mu^{-1}\|(w_{\mu k}^{-1})_k\|_{\ell_{p^*}})_\mu \in \ell_{q_*} $$ 
which proves that $RB^{s_1}_{p_1,q_1}(w_1)\subseteq RB^{s_2}_{p_2,q_2}(w_2)$ if (\ref{CNS-cont-sequ-space}) holds. 

This embedding is compact if moreover 
\begin{equation*}
\begin{split} 
 & \lim_{\mu \to+\infty} \beta_\mu^{-1}\|(w_{\mu k}^{-1})_k\|_{\ell_{p^*}} = 0 \quad \text{ if } q^*=\infty \\
 & \lim_{|k|\to +\infty} w_{\mu k} = \infty \quad \text{ for all $\mu \ge 0 $ if } p^*=\infty.  
\end{split} 
\end{equation*} 
which proves the theorem.
\end{proof}

As an example of application we now consider the case $w_1(x)=|x|^{\gamma_1}$, $w_2(x)=|x|^{\gamma_2}$ with $\gamma_1, \gamma_2 > -n$ so that $w_1,w_2$ are $\mathcal{A}_\infty$-weights. 
In order to simplify the statement of the following examples we introduce
\begin{equation} \label{def-delta} \delta:=s_1 - \frac{n}{p_1}-s_2 + \frac{n}{p_2}. \end{equation}

\begin{ex}
\label{example-besov}
Let $p_1, p_2, q_1, q_2 \in [1, \infty]$ and $\gamma_1, \gamma_2 >-n$. There is a continuous embedding $RB_{p_1,q_1}^{s_1}(\Rn,|x|^{\gamma_1})\to RB_{p_2,q_2}^{s_2}(\Rn,|x|^{\gamma_2})$ provided that
$$
   \begin{cases} \frac{\gamma_1}{p_1}-\frac{\gamma_2}{p_2}  \ge (n-1)\left(\frac{1}{p_2}-\frac{1}{p_1}\right) &\text{ if } p^*=\infty \\
                \frac{\gamma_1}{p_1}-\frac{\gamma_2}{p_2} > \frac{n}{p*}  &\text{ if } p^*<\infty 
   \end{cases} 
   \quad\text{ and } \quad 
 \begin{cases} \delta \ge \frac{\gamma_1}{p_1}-\frac{\gamma_2}{p_2} &\text{ if } q^*=\infty \\
                 \delta > \frac{\gamma_1}{p_1}-\frac{\gamma_2}{p_2} &\text{ if } q^*<\infty 
   \end{cases}
$$ 
where $\delta$ is as in \eqref{def-delta}.
This embedding is compact provided that the previous conditions hold and  moreover
$$ \frac{\gamma_1}{p_1}-\frac{\gamma_2}{p_2}>(n-1)\left(\frac{1}{p_2}-\frac{1}{p_1}\right)\quad\text{ if } p^*=\infty  \quad\text{ and } \quad  
\delta> \frac{\gamma_1}{p_1}-\frac{\gamma_2}{p_2} \quad\text{ if } q^*=\infty. $$
\end{ex}

\begin{proof} 
Since $|x|\sim k2^{-\mu}$ for $x\in A_{\mu k}$, we have for $i=1,2$ that 
$$
w_{\mu k}^i = \|\tilde\chi_{A_{\mu k}}\|_{L^{p_i}(|x|^{\gamma_i})}
              \sim |A_{\mu k}|^{-1/2}((k2^{-\mu})^{\gamma_i}|A_{\mu k}|)^\frac{1}{p_i}. 
$$
Moreover $|A_{\mu k}|\sim k^{n-1}2^{-\mu n}$. Hence 
$$ 
\frac{w_{\mu k}^2}{w_{\mu k}^1} \sim 2^{\mu (\frac{n+\gamma_1}{p_1}-\frac{n+\gamma_2}{p_2})}k^{\frac{\gamma_2}{p_2}-\frac{\gamma_1}{p_1}+(n-1)(\frac{1}{p_2}-\frac{1}{p_1})}. 
$$

Then if e.g. $p^*,q^*<\infty$ then (\ref{CNS-cont-sequ-space}) writes 
$$
\sum_k k^{p^*(\frac{\gamma_2}{p_2}-\frac{\gamma_1}{p_1}+(n-1)(\frac{1}{p_2}-\frac{1}{p_1}))} <\infty  
  \quad \text{ and } \quad 
  \sum_\mu 2^{\mu q^*(\frac{\gamma_1}{p_1}-\frac{\gamma_2}{p_2}-\delta)} <\infty 
 $$ 
i.e. 
$$ p^*\left(\frac{\gamma_2}{p_2}-\frac{\gamma_1}{p_1}+(n-1)\left(\frac{1}{p_2}-\frac{1}{p_1}\right)\right)<-1 \quad \text{ and } \quad  q^*\left(\frac{\gamma_1}{p_1}-\frac{\gamma_2}{p_2}-\delta\right)<0. $$  
Recalling the definition of $p^*,q^*$ this gives the statement. 

Concerning the compactness we have 
$$  
2^{\mu(s_2-s_1)} \left\{\sum_k \left(\frac{w_{\mu k}^2}{w_{\mu k}^1}\right)^{p^*}\right\}^\frac{1}{p^*}
 \sim  2^{\mu(\frac{\gamma_1}{p_1}-\frac{\gamma_2}{p_2}-\delta)} \left\{\sum_k k^{p^*\left(\frac{\gamma_2}{p_2}-\frac{\gamma_1}{p_1}+(n-1)(\frac{1}{p_2}-\frac{1}{p_1})\right)}   \right\}^\frac{1}{p^*} 
$$
where the sum in the right hand side is finite. 

\end{proof}

\begin{remark}
\begin{enumerate}
\item It is immediate form the above example that one has an improvement with respect to the non-radial case, c.f. \cite[Proposition 2.8]{HS1}. Indeed, in the case $p^*=\infty$ (that is, $p_1\le p_2$) we can have $\frac{\gamma_1}{p_1}-\frac{\gamma_2}{p_2} <0$, in which case $\delta$ can be negative as well, while in the non-radial case both values must be non-negative.
\item An  alternative proof of the continuity part of the above example  can be found in \cite[Theorem 12]{DD4}. For the corresponding non-radial case see \cite[Theorem 1.1]{MV}.
\end{enumerate}
\end{remark}

Our next examples concern weights of purely polynomial growth. To this end, let $w_{\alpha,\beta}=\begin{cases}
|x|^\alpha  &\mbox{if } |x|\le 1\\
|x|^\beta &\mbox{if } |x|> 1\\
\end{cases}$ with $\alpha, \beta >-n.$

\begin{ex}
Let $-\infty < s_2 \le s_1 < \infty$, $0< p_1 <\infty$, $0<p_2\le \infty$ and $0<q_1,q_2\le \infty$. Then, there is a continuous embedding $RB_{p_1,q_1}^{s_1}(\Rn,w_{\alpha, \beta})\to RB_{p_2,q_2}^{s_2}(\Rn)$ provided
\begin{equation*}\label{CNS10}
\begin{cases}
 \frac{\beta}{p_1}\ge (n-1)(\frac{1}{p_2}-\frac{1}{p_1}) & \mbox{ if } p^*=\infty\\
 \frac{\beta}{p_1}>\frac{n}{p^*} &\mbox{ if } p^*<\infty
\end{cases}
\end{equation*}
and one of the following conditions is satisfied:
$$
\begin{cases}
\delta \ge \max (\frac{\alpha}{p_1} , (n-1)(\frac{1}{p_2}-\frac{1}{p_1})) &\mbox{ if } q^*=\infty, p^*=\infty \\
\delta > \max( \frac{\alpha}{p_1} , (n-1)(\frac{1}{p_2}-\frac{1}{p_1})) &\mbox{ if }  q^*<\infty, p^*=\infty\\
\delta \ge \max(\frac{\alpha}{p_1}, \frac{n}{p^*}) &\mbox{ if }  q^*=\infty, p^*<\infty, \frac{n}{p*}\neq \frac{\alpha}{p_1}\\
\delta > \max(\frac{\alpha}{p_1}, \frac{n}{p^*}) &\mbox{ otherwise} 
\end{cases}
$$
Moreover the embedding $RB_{p_1,q_1}^{s_1}(\Rn,w_{\alpha, \beta})\to RB_{p_2,q_2}^{s_2}(\Rn)$ is compact provided that
$$
\begin{cases}
 \frac{\beta}{p_1}> (n-1)(\frac{1}{p_2}-\frac{1}{p_1}) & \mbox{ if } p^*=\infty\\
 \frac{\beta}{p_1}>\frac{n}{p^*} &\mbox{ if } p^*<\infty
\end{cases}
$$
and
$$ \begin{cases}
\delta > \max( \frac{\alpha}{p_1} , (n-1)(\frac{1}{p_2}-\frac{1}{p_1})) &\mbox{ if } p^*=\infty\\
\delta > \max(\frac{\alpha}{p_1}, \frac{n}{p^*}) &\mbox{ if }  p^*<\infty \\
\end{cases}
$$
\end{ex}

\begin{proof}

Consider first the Besov case. We have 
$$
\frac{w_{\mu k}^2}{w_{\mu k}^1} \sim k^{(n-1)(\frac{1}{p_2}-\frac{1}{p_1})} \, 2^{-\mu n (\frac{1}{p_2}-\frac{1}{p_1})} \times \begin{cases}
k^{-\frac{\alpha}{p_1}} \, 2^{ \frac{\mu \alpha}{p_1}} & \mbox{ if }k\le 2^\mu \\
k^{-\frac{\beta}{p_1}} \, 2^{ \frac{\mu \beta}{p_1}} & \mbox{ if }k> 2^\mu 
\end{cases}
$$
Then if e.g. $p^*=\infty, q^*<\infty$,  \eqref{CNS-cont-sequ-space} writes
$$
\sum_\mu 2^{\mu q^*[(s_2-s_1) -n (\frac{1}{p_2}-\frac{1}{p_1})+ \frac{\alpha}{p_1}]} \left(\sup_{k\le 2^\mu} k^{(n-1)(\frac{1}{p_2}-\frac{1}{p_1})-\frac{\alpha}{p_1}}\right)^{q^*} <\infty
$$
and
$$
\sum_\mu 2^{\mu q^*[(s_2-s_1) -n (\frac{1}{p_2}-\frac{1}{p_1})+ \frac{\beta}{p_1}]} \left(\sup_{k> 2^\mu} k^{(n-1)(\frac{1}{p_2}-\frac{1}{p_1})-\frac{\beta}{p_1}}\right)^{q^*} <\infty
$$
which gives the statement. As for the compactness, we need that 
$$
\lim_{|k|\to \infty}k^{(n-1)(\frac{1}{p_2}-\frac{1}{p_1})-\frac{\beta}{p_1}} = 0.
$$
 The remaining cases are analogous. 
 
\end{proof}

The generalization to the following two-weighted embeddings is straightforward:

\begin{ex}
Let $-\infty < s_2 \le s_1 < \infty$, $0< p_1 <\infty$, $0<p_2\le \infty$ and $0<q_1,q_2\le \infty$. Then, there is a continuous embedding $RB_{p_1,q_1}^{s_1}(\Rn,w_{\alpha_1, \beta_1})\to RB_{p_2,q_2}^{s_2}(\Rn,w_{\alpha_2, \beta_2})$ provided
$$
\begin{cases}
 \frac{\beta_1}{p_1}-\frac{\beta_2}{p_2}\ge (n-1)(\frac{1}{p_2}-\frac{1}{p_1}) & \mbox{ if } p^*=\infty\\
 \frac{\beta_1}{p_1}-\frac{\beta_2}{p_2}>\frac{n}{p^*} &\mbox{ if } p^*<\infty
\end{cases}
$$
and one of the following conditions is satisfied:
$$
\begin{cases}
\delta \ge \max (\frac{\alpha_1}{p_1}-\frac{\alpha_2}{p_2} , (n-1)(\frac{1}{p_2}-\frac{1}{p_1})) &\mbox{ if } q^*=\infty, p^*=\infty \\
\delta > \max( \frac{\alpha_1}{p_1}-\frac{\alpha_2}{p_2} , (n-1)(\frac{1}{p_2}-\frac{1}{p_1})) &\mbox{ if }  q^*<\infty, p^*=\infty\\
\delta \ge \max(\frac{\alpha_1}{p_1}-\frac{\alpha_2}{p_2}, \frac{n}{p^*}) &\mbox{ if }  q^*=\infty, p^*<\infty, \frac{n}{p*}\neq \frac{\alpha_1}{p_1}-\frac{\alpha_2}{p_2}\\
\delta > \max(\frac{\alpha_1}{p_1}-\frac{\alpha_2}{p_2}, \frac{n}{p^*}) &\mbox{ otherwise} 
\end{cases}
$$
where $\delta$ is as in \eqref{def-delta}.

Moreover the embedding $RB_{p_1,q_1}^{s_1}(\Rn,w_{\alpha_1, \beta_1})\to RB_{p_2,q_2}^{s_2}(\Rn,w_{\alpha_2, \beta_2})$, $p_1, p_2 \in (0,\infty)$, is compact provided that

$$
\begin{cases}
 \frac{\beta_1}{p_1}-\frac{\beta_2}{p_2}> (n-1)(\frac{1}{p_2}-\frac{1}{p_1}) & \mbox{ if } p^*=\infty\\
 \frac{\beta_1}{p_1}-\frac{\beta_2}{p_2}>\frac{n}{p^*} &\mbox{ if } p^*<\infty
\end{cases}
$$
and one of the following conditions is satisfied:
$$
\begin{cases}
\delta > \max (\frac{\alpha_1}{p_1}-\frac{\alpha_2}{p_2} , (n-1)(\frac{1}{p_2}-\frac{1}{p_1})) &\mbox{ if } p^*=\infty \\
\delta > \max(\frac{\alpha_1}{p_1}-\frac{\alpha_2}{p_2}, \frac{n}{p^*}) &\mbox{ if }  p^*<\infty
\end{cases}
$$

\end{ex}

\section{Continuous and compact embeddings of weighted radial Triebel-Lizorkin spaces}
\label{section-TL}

Our next result concers embeddings for Triebel-Lizorkin spaces with radial $\mathcal{A}_\infty$ weights. We will follow the approach in \cite{MV2}, which is based on a Gagliardo-Nirenberg type inequality and two lemmas on products of Muckenhoupt weights that we recall for the reader's convenience.

\begin{proposition}\cite[Proposition 5.1]{MV} 
Let $q, q_0, q_1 \in [1, \infty]$ and $\theta \in (0,1)$. Let $p, p_0, p_1 \in (1, \infty)$ and $-\infty<s_0<s_1<\infty$ satisfy
$$
\frac{1}{p}=\frac{1-\theta}{p_0}+\frac{\theta}{p_1} \;  \mbox{ and } \; s = (1-\theta)s_0+ \theta s_1.
$$
Let further $w, w_0, w_1 \in \mathcal{A}_\infty$ be such that $w=w_0^{(1-\theta)p/p_0}w_1^{\theta p/p_1}.$ Then there exists a constant $C$ such that for all $f\in \mathcal{S}^\prime(\R^n)$ one has
$$
\|f\|_{F^s_{p,q}(w)} \le C \|f\|_{F^{s_0}_{p_0,q_0}(w_0)}^{1-\theta}  \|f\|_{F^{s_1}_{p_1,q_1}(w_1)}^\theta.
$$
\label{propositionMV}
\end{proposition}

\begin{lemma}\cite[Lemma 3.1]{MV2}
\label{MV-Ap}
Let $1<p<\infty$ and $w_1, w_2 \in \mathcal{A}_p$. Then, there is $\eta_0>0$ such that, for all $\varepsilon, \delta \in [0,\eta_0)$ one has $w_1^{-\varepsilon}w_2^{1+\delta} \in \mathcal{A}_p$.
\end{lemma}

\begin{lemma}\cite[Lemma 3.2]{MV2}
\label{MV-cubos}
Let $w_1, w_2 \in \mathcal{A}_\infty$. Then there are $\eta_0>0$ and a constant $C>0$ such that for all $\varepsilon, \delta \in (0,\eta_0)$ and all cubes $Q\subset \mathbb{R}^n$ we have
$$
\int_Q w_1^{-\varepsilon} w_2^{1+\delta} \, dx \le C |Q|^{\varepsilon-\delta} \left(\int_Q w_1 \, dx \right)^{-\varepsilon} \left(\int_Q w_2 \, dx \right)^{1+\delta}.
$$
\end{lemma}

Since our functions are supported on annuli instead of cubes, we will need another auxiliary lemma on the behavior of products of radial Muckenhoupt weights over these sets. To this end, we first recall the following characterization of radial $\mathcal{A}_p$ weights given by Duoandikoetxea et al. in \cite{D}:

\begin{lemma}\cite[Theorem 3.2]{D}
\label{Ap-radiales}
Let $w_0: (0,\infty) \to [0,\infty]$ and $w_n(x)=w_0(|x|)$ for $x\in \mathbb{R}^n$. Then $w_n$ is in $\mathcal{A}_p(\mathbb{R}^n)$ if and only if $\delta_nw_0$ is in $\mathcal{A}_p(0,+\infty)$, where $\delta_n w_0(t)=w_0(t^{1/n}).$
\end{lemma}

\begin{lemma} 
\label{lema-anillos}
Let $w_1, w_2 \in \mathcal{A}_\infty$, $w_1(x)=\tilde w_1(|x|), w_2(x)=\tilde w_2(|x|)$ for all $x\in \mathbb{R}^n$. Then, there exists $\eta_0>0$ such that for all $\varepsilon \in (0,\eta_0)$ and any annnulus $D_{ab}=\{x\in \mathbb{R}^n : a\le |x|\le b\}, a,b \in \mathbb{R}_+$,
$$
\int_{D_{ab}} w_1^{-\varepsilon} w_2^{1+\varepsilon} \, dx \le C  \left(\int_{D_{ab}} w_1 \, dx\right)^{-\varepsilon} \left(\int_{D_{ab}} w_2 \, dx\right)^{1+\varepsilon} 
$$
\end{lemma}

\begin{proof}
Fix $p>1$ such that  $w_1, w_2 \in \mathcal{A}_p$, let $\eta_0$ be as in Lemma \ref{MV-cubos} and $\varepsilon \in (0, \eta_0)$. Taking polar coordinates we obtain
\begin{align*}
\int_{D_{ab}} w_1^{-\varepsilon} w_2^{1+\varepsilon} \, dx  &= \omega_{n-1} \int_a^b \tilde w_1^{-\varepsilon} \,  \tilde w_2^{1+\varepsilon} r^{n-1} \, dr \\
&= \omega_{n-1} \int_{a^n}^{b^n} (\delta_n \tilde w_1)^{-\varepsilon} \,  (\delta_n \tilde w_2)^{1+\varepsilon}\, \frac{dr}{n}\\
&\le C \omega_{n-1}  \left( \int_a^b \delta_n \tilde w_1 \, dr \right)^{-\varepsilon}  \left( \int_a^b \delta_n \tilde w_2 \, dr \right)^{1+\varepsilon}
\end{align*}
where the last bound follows from Lemma \ref{MV-cubos}, and we have used the fact that $\delta_n \tilde w_1, \delta_n \tilde w_2 \in \mathcal{A}_p(0,+\infty)$ by Lemma \ref{Ap-radiales}. Changing variables again we obtain the desired bound.
\end{proof}

Now we are ready to prove our result for Triebel-Lizorkin spaces: 

\begin{theorem}
\label{main-theorem-TL} Let $1\le p_1\le p_2\le \infty$, $q_1, q_2\in[1, \infty]$ and $w_1,w_2$ be radially symmetric $\mathcal{A}_\infty$-weights.   
There is a continuous embedding $RF_{p_1,q_1}^{s_1}(\Rn, w_1)\to RF_{p_2,q_2}^{s_2}(\Rn, w_2)$ provided that
\begin{equation}
\label{cond-TL}
\sup_{\mu, k} \left\{2^{-\mu(s_1-s_2)} \frac{w^2_{\mu k}}{w^1_{\mu k}} \right\}< +\infty
\end{equation}

where 
$$ w^1_{\mu k} = \|\tilde\chi_{A_{\mu k}}\|_{L^{p_1}(\Rn, w_1)},\quad 
  w^2_{\mu k} = \|\tilde\chi_{A_{\mu k}}\|_{L^{p_2}(\Rn, w_2)}, $$ 

The embedding is compact provided that \eqref{cond-TL} holds and
\begin{equation*}
\begin{split} 
 & \lim_{|k|\to +\infty} \frac{w^1_{\mu k}}{w^2_{\mu k}} = \infty \quad \text{ for all $\mu\ge 0$}.  
\end{split} 
\end{equation*} 

\end{theorem} 
\begin{proof}
The proof  has two steps: proving the continuity of the embedding and then the compactness. 

For the first part, we follow closely the approach in \cite{MV2}, which we outline for the reader's convenience. Note that it suffices to prove the continuity of the embedding $RF^{s_1}_{p_1, q_1}(\R^n, w_1) \hookrightarrow RF^{ s_2}_{p_2, q_2}(\R^n, w_2)$ with $q_2 \le \min\{p_1, p_2\}$ since then the result follows from Theorem \ref{elementary-embeddings}(3). 

Since $p_2\ge p_1$, there exists $\theta_0 \in [0,1)$ such that $\frac{1}{p_2}-\frac{1-\theta_0}{p_1}=0$ (in fact, $\theta_0=\frac{p_2}{p_1}-1$). For $\theta\in (\theta_0,1)$, let 
$$\varepsilon=\frac{\frac{1-\theta}{p_1}}{\frac{1}{p_2}-\frac{1-\theta}{p_1}}>0$$
which clearly tends to zero as $\theta \to 1$, and let $v,r,t$ be defined by the identities
$$
v=w_1^{-\varepsilon}w_2^{1+\varepsilon}, \qquad \frac{1}{p_2}=\frac{1-\theta}{p_1}+\frac{\theta}{r}, \qquad s_2=(1-\theta)s_1+\theta t.
$$
Then, one can check that $w_2=w_1^{(1-\theta)p_2/p_1} v^{p_2\theta/r}$, $r\in [p_2, +\infty)$ and $t<s_2<s_1$. Moreover, $v\in A_{p_2}\subseteq A_r$ if $\theta$ is sufficiently close to 1 by Lemma \ref{MV-Ap}. 
Hence, by Proposition \ref{propositionMV}, it holds
\begin{equation}
\label{cota-interpolacion}
\|f\|_{RF_{p_2,q_2}^{s_2}(\Rn, w_2)} \le C \|f\|_{RF_{p_1,q_1}^{s_1}(\Rn, w_1)}^{1-\theta} \|f\|_{RF_{r,r}^{t}(\Rn, v)}^\theta
\end{equation}

Now, since $RB_{p,p}^s=RF_{p,p}^s$ and $r\ge p_2$,  by Theorem \ref{main-theorem} 
$$
\|f\|_{RF_{r,r}^{t}(\Rn, v)}\le C \|f\|_{RF_{p_2,p_2}^{s_2}(\Rn, w_2)}
$$
holds provided that 
\begin{equation}
\label{cond-integrales}
\sup_{k,\mu} 2^{-\mu(s_2-t)} \left( \int_{A_{\mu k}} v \, \right)^{1/r} \left( \int_{A_{\mu k}} w_2\right)^{1/p_2} < +\infty. 
\end{equation}
But, by Lemma \ref{lema-anillos}
 $$
 \int_{A_{\mu k}} v \, \le C  \left(\int_{A_{\mu k}} w_1\right)^{-\varepsilon}\left(\int_{A_{\mu k}} w_2\right)^{1+\varepsilon}
 $$
 whence, replacing this bound into \eqref{cond-integrales} and noting that $s_2-t=(s_1-s_2)\frac{1-\theta}{\theta}$, $\frac{\varepsilon}{r}=\frac{1}{p_1}\frac{1-\theta}{\theta}$, and $\frac{1+\varepsilon}{r}=\frac{1}{\theta p_2}$, the desired bound finally follows from condition \eqref{cond-TL}.
 
 Since $q_2\le p_2$ by the above assumption, we may replace $RF^{s_2}_{p_2,p_2}$ on the right hand side of \eqref{cota-interpolacion} by $RF^{s_2}_{p_2,q_2}$, and divide by $\|f\|_{RF_{p_2,q_2}^{s_2}(\Rn, w_2)}^\theta$ to obtain the bound
 
 $$
 \|f\|_{RF_{p_2,q_2}^{s_2}(\Rn, w_2)} \le C \|f\|_{RF_{p_1,q_1}^{s_1}(\Rn, w_1)}.
$$
Notice that, in principle, this bound holds in the intersection $RF_{p_1,q_1}^{s_1}(\Rn, w_1) \cap RF_{p_2,q_2}^{s_2}(\Rn, w_2)$, but it can be extended by density to $RF_{p_1,q_1}^{s_1}(\Rn, w_1)$ (see \cite[Proof of Theorem 1.2]{MV2}).
 
 It remains to prove that the embedding is compact. To this end, let $(f_k)_{k\in \N}$ be such that $\|f_k\|_{RF_{p_1,q_1}^{s_1}(w_1)} \le C$. Then, by the embedding we have already proved,  $(f_k)_{k\in\N}$ is also bounded in $RF^{s_2}_{p_2,q_2}(w_2)$ with $q_2 \le \min\{p_1, p_2\}$ and, by Theorem \ref{elementary-embeddings} in $RB^{s_2}_{p_2,p_2}(w_2)$. Since, under our hypotheses, the embedding $RB^t_{r,r}(v) \hookrightarrow RB^{s_2}_{p_2,p_2}(w_2)$ is compact by Theorem \ref{main-theorem}, we have that $f_k \to f$ in  $B_{r,r}^t=F_{r,r}^t$. Then, 
 $$
\|f_k - f\|_{RF^{s_2}_{p_2,q_2}(w_2)} \le \|f_k - f\|_{RF^{s_1}_{p_1,q_1}(w_1)}^{1-\theta} \|f_k - f\|_{RF^{t}_{r,r}(v)}^\theta   \to 0 $$
which proves our statement.

\end{proof}

Examples for the same weights considered in the Besov case can be obtained in an analogous manner. We leave the proofs to the reader. 

An interesting special case of the inhomogenous Triebel-Lizorkin spaces is given by the Bessel potential spaces. In  \cite{DD6} the first two authors proved (with a more elementary argument) the following   result.

\begin{ex}\cite[Theorems 6.4 and 7.2]{DD6}
Let $1<p<\infty$, $0<s<\frac{n}{p}$, $p \leq q \leq p^*_c = \frac{p(n+c)}{n-sp}$.
Then we have a continuous embedding
\be
H^{s,p}_{rad} (\mathbb{R}^n) \subset L^q(\mathbb{R}^n,|x|^c dx)
\label{our-embedding}
\ee
provided that
\be
-sp < c < \frac{(n-1)(q-p)}{p}
\label{cota-del-peso}
\ee
Morevover, the embedding is compact when $p<q<p^*_c$.
\end{ex}

\begin{proof}
To see this result as a special case of the  embeddings  in Theorem \ref{main-theorem-TL},  notice that $H^{s,p}_{rad}=RF^s_{p,2}$ and $L_{rad}^q(|x|^c)=RF^0_{q,2}(|x|^c)$ provided $|x|^c \in A_q$ (that is, $-n<c<n(q-1))$. Hence, this case corresponds to the choice $w_1=1$, $w_2=|x|c$, $p_1=p$, $q_1=2$, $p_2=q$, and $q_2=2$. Moreover, since we are interested in the case $q\ge p$, this implies $p^*=\infty$, while $q^*=\infty$ by the choice of spaces. Therefore, we obtain $c< \frac{(n-1)(q-p)}{p}$ and $q<\frac{p(c+n)}{n-sp}$. The remaining conditions $c>-sp$ and $s<\frac{n}{p}$ are needed to have a non-empty interval of admissible values of $q$.
\end{proof}

A different proof of the previous example for $p=2$ was also given in \cite{DDD2} by the first two authors jointly with R. Dur\'an, where that result was used to analyze the existence of radial solutions of a weighted elliptic system with hamiltonian structure in $\Rn$.

\end{document}